\documentclass[11pt,letterpaper]{amsart}
\usepackage{amsmath, amssymb}
\usepackage[margin=1in]{geometry}
\title[Limit of random walks with geometric area tilts]{Scaling limit for line ensembles of random walks with geometric area tilts}
\author{Christian Serio}
\address{Department of Mathematics, Stanford University, Stanford, CA 94305, USA}
\email{cdserio@stanford.edu}
\newtheorem{theorem}{Theorem}

\newtheorem{lemma}{Lemma}[section]

\theoremstyle{remark}
\newtheorem{remark}{Remark}
\usepackage{verbatim}
\usepackage{hyperref}
\hypersetup{
	colorlinks   = true, 
	urlcolor     = blue, 
	linkcolor    = blue, 
	citecolor   = blue 
}

\begin{document}
	
	\maketitle
	
	\begin{abstract}
		We consider line ensembles of non-intersecting random walks constrained by a hard wall, each tilted by the area underneath it with geometrically growing pre-factors $\mathfrak{b}^i$ where $\mathfrak{b}>1$. This is a model for the level lines of the $(2+1)\mathrm{D}$ SOS model above a hard wall, which itself mimics the low-temperature 3D Ising interface. A similar model with $\mathfrak{b}=1$ and a fixed number of curves was studied by Ioffe, Velenik, and Wachtel (2018), who derived a scaling limit as the time interval $[-N,N]$ tends to infinity. Line ensembles of Brownian bridges with geometric area tilts ($\mathfrak{b}>1$) were studied by Caputo, Ioffe, and Wachtel (2019), and later by Dembo, Lubetzky, and Zeitouni (2022+). Their results show that as the time interval and the number of curves $n$ tend to infinity, the top $k$ paths converge to a limiting measure $\mu$. In this paper we address the open problem of proving existence of a scaling limit for random walk ensembles with geometric area tilts. We prove that with mild assumptions on the jump distribution, under suitable scaling the top $k$ paths converge to the same measure $\mu$ as $N\to\infty$ followed by $n\to\infty$. We do so both in the case of bridges fixed at $\pm N$ and of walks fixed only at $-N$.
	\end{abstract}

	\section{Introduction}\label{results}
	
	\subsection{Background and motivation}\label{background}
	
	The geometry of low-temperature interfaces in two- and three-dimensional lattice models has been the subject of much interest in statistical physics for several decades. In the 3D Ising model on a cylinder at low temperature with Dobrushin boundary conditions, the interface can be viewed as a two-dimensional random surface (possibly with self-intersections) which is zero on the boundary and separates $+$ and $-$ spins in the configuration. In the absence of a wall, the interface fluctuates near height zero \cite{Dob}. When conditioned to stay above a hard wall, the interface experiences \textit{entropic repulsion} and is expected to become rigid at a height diverging with the diameter $L$ of the cylinder \cite{GL}.
	
	This behavior was first confirmed in \cite{SOS0} for the $(2+1)$-dimensional solid-on-solid (SOS) model above a hard wall, an approximation of the 3D Ising interface defined by a random nonnegative height function on an $L\times L$ box in $\mathbb{Z}^2$. There it was shown that in the bulk, the surface is typically propelled to height $H(L) \sim \log L$. This was extended in \cite{SOS1, SOS2} with a detailed description of the geometry of the SOS surface: the surface is characterized by a unique ensemble of $n=H(L)$ nested contours, which can be viewed as the level lines of the surface. As $L\to\infty$ with the box rescaled to the unit square, these contours were shown to converge to a unique limit shape of infinitely many nested loops. In the bulk away from the corners of the box, these loops all lie flat against the sides of the box. 
	
	A natural question is then to study the fluctuations of the level lines from the sides of the box in the bulk. It is known from \cite{SOS2} that for the top line, these fluctuations are of order $L^{1/3+o(1)}$ on middle portions of the boundary of length linear in $L$. The level lines in this region may be approximated by $n$ ordered height functions $\psi_1 \geq \cdots \geq \psi_n \geq 0$ on $\{-L,\dots,L\}$, and it was shown via cluster expansion in \cite{SOS2} that the probabilistic weight of these random paths includes \textit{geometrically growing area tilts} of the form $\exp(-\mathfrak{b}^i L^{-1}A(\psi_i))$ for each $i$, where $\mathfrak{b}>1$ is a constant and $A(\psi_i)=\sum_j \psi_i(j)$ represents the area under $\psi_i$. We refer to \cite{SOS1, SOS2, SOS3} for further exposition on contours of the SOS model above a wall and detailed statements and proofs of these facts, and to \cite{CIW19} for further motivation of this model of random height functions with area tilts.
	
	Motivated by this analysis, in this paper we consider scaling limits as $N,n\to\infty$ of a line ensemble of $n$ random walks on an interval $[-N,N]$, conditioned to remain above zero and not to intersect one another, and each tilted by the area underneath it with a geometrically growing prefactor. We give a precise definition of this model in the next section, but before doing so we will review some previous work on similar models of SOS level lines and related two-dimensional interfaces. These models fit into the broad framework of \textit{Gibbsian line ensembles}, which have notably been studied in \cite{CH14, CH16} and numerous subsequent works on models mostly lying in the KPZ universality class. Although the physical motivations for the models we study here are somewhat different, the Gibbsian techniques developed in these works remain fundamental to the analysis.
	
	A model of the top SOS level line, consisting of a random walk above a wall tilted by its area, was studied in \cite{ISV} and shown to possess as a scaling limit the Ferrari-Spohn diffusion. Subsequently, \cite{IVW} studied a fixed number $n$ of non-intersecting random walk bridges above a wall with area tilts \textit{without} geometric pre-factors, i.e., in the case $\mathfrak{b}=1$ above. They proved convergence to an explicit scaling limit, the Dyson Ferrari-Spohn diffusion, a determinantal process of $n$ non-intersecting Ferrari-Spohn diffusions. Their argument relies on two main inputs: a mixing bound used to prove tightness, and the Karlin-McGregor formula for finite-dimensional convergence. For the latter point the choice $\mathfrak{b}=1$ is essential, as the determinantal structure is lost when $\mathfrak{b}>1$. For further discussion of Ferrari-Spohn scaling limits of interface models, we refer to the survey \cite{IV}, as well as the recent works \cite{FS, DS} which show that the Dyson Ferrari-Spohn diffusion has the Airy line ensemble as a scaling limit as $n\to\infty$.
	
	Geometrically growing area tilts with a diverging number of curves were first treated in \cite{CIW19, CIW2} for a Brownian polymer model, with the random walk bridges replaced by $n$ Brownian bridges on $[-T,T]$. Unlike in the non-geometric tilt case $\mathfrak{b}=1$, where tightness is not expected as $n\to\infty$, the additional effect of the growing prefactors in the $\mathfrak{b}>1$ case allows for a proof of tightness as $T,n\to\infty$ for the Brownian polymer with both zero and free boundary conditions (with no scaling). For zero boundary conditions, convergence to an as-yet-unidentified limiting diffusion $\mu$ was also proven using stochastic monotonicity. The arguments used in \cite{CIW19, CIW2} rely fundamentally on the scaling invariance of Brownian motion, and do not readily extend to the rescaled random walk case. Convergence to the limit $\mu$ was extended to the Brownian polymer with free boundary conditions (with $T\to\infty$ first and then $n\to\infty$) in \cite{DLZ} using a spectral theory approach. This was recently improved in \cite{CG} to allow $T,n\to\infty$ in any order; moreover, mixing (hence ergodicity) and tail estimates for the limiting process were established.
	
	In this paper, we expand upon the results above by showing that line ensembles of $n$ random walks on $[-N,N]$ with geometric area tilts converge after suitable diffusive scaling (including 1:2:3 scaling as a special case) to the same limiting diffusion $\mu$ as the Brownian polymer, as $N\to\infty$ followed by $n\to\infty$, addressing the open question posed in \cite[3.5.3]{CIW2}. We do this both for random walk bridges and for random walks which are fixed only at $-N$ and free at $N$. Our approach is essentially to prove an invariance principle for the random walk ensembles on a fixed interval towards the Brownian polymers. Then for bridges we exploit the mixing bounds provided by \cite{IVW} and the convergence result of \cite{CIW2} for the zero-boundary Brownian polymer. We deal with walks (free at $N$) by extending the mixing bounds from \cite{IVW} to this setting, and using the stochastic monotonicity results of \cite{CIW19} and the convergence of the free-boundary Brownian polymer from \cite{DLZ}.

	\subsection{Definitions}
	
	We will now define the line ensemble models we will study and review known results in greater detail. Our notation is mostly a mixture of that in \cite{DLZ} and \cite{IVW}, and in particular the definition of our model is meant to mirror that in \cite{IVW}. Throughout this paper, for a given probability measure $\mathbb{P}$ and a functional $F$, we will use $\mathbb{P}[F(X)]$ to denote the expectation of $F(X)$ with respect to $\mathbb{P}$, where $X$ is a random variable with law $\mathbb{P}$. We will not specify the law of $X$ when it is clear from context. We will be considering weak convergence of measures in the topology of uniform convergence on compact sets, which we will abbreviate by (u.c.c.). 
	
	We let $n\in\mathbb{N}$ denote the number of curves, which will be fixed unless stated otherwise, and we fix parameters $\mathfrak{a}>0$ and $\mathfrak{b} > 1$. Let
	\[
	\mathbb{A}_n^+ = \{\underline{x}\in\mathbb{R}^n : x_1 > \cdots > x_n > 0\}, \quad \mathbb{A}_n^0 = \{\underline{x}\in\mathbb{R}^n : x_1 \geq \cdots \geq x_n \geq 0\}
	\]
	denote the open and closed Weyl chambers in $\mathbb{R}^n$. 
	
	Let $(p_z)_{z\in\mathbb{Z}}$ be an irreducible random walk kernel with mean 0, variance 1, and finite exponential moments as in \cite{IVW} for simplicity (although this last assumption is probably not strictly necessary). For $u,v,M,N\in\mathbb{N}$ with $M<N$, let $\mathbb{P}^u_M$ denote the law of a random walk on $\{M,M+1,\dots\}$ with kernel $p$ starting at $u$ at time $M$, and let $\mathbb{P}^{u,v}_{M,N}$ denote the law of a random walk bridge on $\{M,\dots,N\}$ with kernel $p$ starting at $u$ at $M$ and ending at $v$ at $N$. For $\underline{u},\underline{v}\in\mathbb{N}^n$, let $\mathbb{P}^{\underline{u}}_M$ denote the law $\mathbb{P}_M^{u_1}\otimes\cdots\otimes\mathbb{P}_M^{u_n}$, and let $\mathbb{P}^{\underline{u},\underline{v}}_{M,N}$ denote $\mathbb{P}^{u_1,v_1}_{M,N} \otimes \cdots \otimes \mathbb{P}^{u_n,v_n}_{M,N}$.
	
	Fix a family of potentials $V_\lambda : [0,\infty) \to [0,\infty)$ for $\lambda > 0$ which are continuous, monotone increasing, and satisfy $V_\lambda(0) = 0$, $\lim_{x\to\infty} V_\lambda(x) = \infty$. For the main results we will work with the most relevant case of the linear potential $V_\lambda(x) = \lambda x$, but some of our auxiliary results hold in this more general setting, and we expect the main results to extend. For $n$ trajectories $X_i = (X_i(M), \dots, X_i(N))\in \mathbb{N}^{N-M+1}$, $1\leq i\leq n$, define the area functional
	\[
	\mathcal{A}^\lambda_{M,N}(\underline{X}) = \mathfrak{a} \sum_{i=1}^n \mathfrak{b}^{i-1} \sum_{j=M}^{N-1} V_\lambda(X_i(j)).
	\]
	Let $\Omega_{M,N}^n(\underline{X})$ denote the event that $\underline{X}(j) = (X_1(j),\dots,X_n(j)) \in \mathbb{A}_n^+$ for all $j\in\{M,\dots,N\}$. We then define the line ensemble measures $\mathbb{P}^{\underline{u}}_{M,N,+,\lambda}$ (walks) and $\mathbb{P}^{\underline{u},\underline{v}}_{M,N,+,\lambda}$ (bridges) via
	\begin{align}
	\label{walk} \mathbb{P}^{\underline{u}}_{M,N,+,\lambda}[F(\underline{X})] &= \frac{1}{Z^{\underline{u}}_{N,+,\lambda}} \mathbb{P}^{\underline{u}}_M \left[ F(\underline{X}) \mathbf{1}_{\Omega_{M,N}^n(\underline{X})} e^{-\mathcal{A}^\lambda_{M,N}(\underline{X})} \right],\\
	\label{bridge} \mathbb{P}^{\underline{u},\underline{v}}_{M,N,+,\lambda}[F(\underline{X})] &= \frac{1}{Z^{\underline{u},\underline{v}}_{N,+,\lambda}} \mathbb{P}^{\underline{u},\underline{v}}_{M,N} \left[ F(\underline{X}) \mathbf{1}_{\Omega_{M,N}^n(\underline{X})} e^{-\mathcal{A}^\lambda_{M,N}(\underline{X})} \right],
	\end{align}
	for any bounded functional $F$ on $\mathbb{N}^{N-M+1}$, with the partition functions 
	\[
	Z^{\underline{u}}_{M,N,+,\lambda} = \mathbb{P}^{\underline{u}}_M \left[ \mathbf{1}_{\Omega_{M,N}^n(\underline{X})} e^{-\mathcal{A}^\lambda_{M,N}(\underline{X})} \right],\quad Z^{\underline{u},\underline{v}}_{M,N,+,\lambda} = \mathbb{P}^{\underline{u}}_{M,N} \left[ \mathbf{1}_{\Omega_{M,N}^n(\underline{X})} e^{-\mathcal{A}^\lambda_{M,N}(\underline{X})} \right].
	\]
	We will most often consider the symmetric case $M=-N$ (although all results readily extend to the asymmetric case), in which case we omit the first subscript and write $\mathbb{P}^{\underline{u}}_{N,+,\lambda}$, etc., for brevity. 

	\begin{remark}\label{gibbs}
	Let us note an important property of these line ensemble measures, the so-called \textit{Gibbs property}. For any integers $M<K<L<N$, it is easy to see by splitting the sum in the area tilt that, conditional on the values of $\underline{X}(j)$ for each $j\in\{-N,\dots,N\}\setminus\{K+1,\dots,L-1\}$, the law of $\underline{X}|_{\{K,\dots,L\}}$ under both $\mathbb{P}^{\underline{u}}_{M,N,+,\lambda}$ and $\mathbb{P}^{\underline{u},\underline{v}}_{M,N,+,\lambda}$ is simply $\mathbb{P}^{\underline{X}(K),\underline{X}(L)}_{K,L,+,\lambda}$. In particular the conditional law only depends on $\underline{X}(K)$ and $\underline{X}(L)$. This property and its Brownian analogue are key tools in the arguments used in \cite{IVW, CIW2} to prove mixing, tightness, and convergence for zero boundary conditions, all of which we use here. The Gibbs property was first used systematically to prove limiting results for line ensembles in \cite{CH14, CH16}; we refer to these papers for further exposition. For the line ensembles we consider here, the Gibbsian structure is used in the proof of the mixing result Theorem \ref{mixing}.
	\end{remark}
	
	Now we establish the scaling we will use. For $\lambda>0$, let $H_\lambda>0$ be the unique number satisfying $H_\lambda^2 V_\lambda(H_\lambda) = 1$. We assume that $\lim_{\lambda\downarrow 0} H_\lambda = \infty$, and that there exist $\lambda_0$ and a continuous non-decreasing function $q_0\geq 0$ on $(0,\infty)$ with $\lim_{r\to\infty} q_0(r) = \infty$ such that for all $\lambda\leq\lambda_0$, 
	\begin{equation}\label{q0}
	H_\lambda^2 V_\lambda(rH_\lambda) \geq q_0(r). 
	\end{equation}
	In particular, in the case $V_\lambda(x) = \lambda x$, we have $H_\lambda = \lambda^{-1/3}$ and we may take $q_0(r) =  r$.
	
	We write $h_\lambda = H_\lambda^{-1}$, $\mathbb{N}_\lambda = h_\lambda\mathbb{N}$, $\mathbb{A}_{n,\lambda}^+ = \mathbb{A}_n^+ \cap \mathbb{N}_\lambda$, and $\mathbb{Z}_\lambda = h_\lambda^2\mathbb{Z}$. For $t\in h_\lambda^2\mathbb{Z}$, define the rescaling 
	\begin{equation}\label{rescaling}
	\underline{x}^\lambda(t) = h_\lambda \underline{X}(H_\lambda^2 t),
	\end{equation}
	and extend to $t\in\mathbb{R}$ by linear interpolation. We will now adjust our notation to this scaling as follows. For $\underline{u},\underline{v}\in\mathbb{A}_{n,\lambda}^+$, $a,b\in\mathbb{Z}_\lambda$ with $a<b$, and $T>0$, we write for brevity
	\[
	\mathbb{P}^{\underline{u}}_{a,b,+,\lambda} := \mathbb{P}^{H_\lambda\underline{u}}_{H_\lambda^2 a, H_\lambda^2 b,+, \lambda}, \quad \mathbb{P}^{\underline{u},\underline{v}}_{a,b,+,\lambda} := \mathbb{P}^{H_\lambda\underline{u},H_\lambda\underline{v}}_{H_\lambda^2 a, H_\lambda^2 b,+,\lambda},
	\]
	and likewise for the partition functions. This should not create any confusion since $\underline{u},\underline{v},a,b$ are generally not integer-valued and this is the only reasonable way to interpret the notation. We write $\mathbb{P}^{\underline{u};T}_{a,b,+,\lambda}$ and $\mathbb{P}^{\underline{u},\underline{v};T}_{a,b,+,\lambda}$ to denote the laws of $\underline{x}^\lambda$ restricted to $[-T,T]$ under these measures. As before if $a=-b$ we omit the first subscript. We will write $\mathcal{A}^\lambda_a(\underline{x}^\lambda) = \mathcal{A}^\lambda_{H_\lambda^2 a}(\underline{X})$; that is, in terms of the rescaled process,
	\begin{equation}\label{arearescaled}
	\mathcal{A}^\lambda_a(\underline{x}^\lambda) = \mathfrak{a}\sum_{i=1}^n \mathfrak{b}^{i-1}\sum_{j=-H_\lambda^2 a}^{H_\lambda^2a - 1} V_\lambda(H_\lambda x_i^\lambda(h_\lambda^2j)).
	\end{equation}
	Finally, we establish notation for the Brownian polymers which will serve as the scaling limits. We let $\mathbf{B}^{\underline{u},\underline{v}}_M$ denote the \textit{unnormalized} path measure of $n$ independent Brownian bridges on $[-M,M]$ with boundary conditions $\underline{u}$ at $-M$ and $\underline{v}$ at $M$, with total mass $(4\pi M)^{-d/2}\exp(-\lVert\underline{u}-\underline{v}\rVert^2/4M)$. For $\underline{u},\underline{v}\in\mathbb{A}_n^+$, define $\mathbb{P}^{\underline{u},\underline{v}}_{M,+,0}$ by 
	\begin{equation}\label{bpoly}
	\mathbb{P}^{\underline{u},\underline{v}}_{M,+,0} \left[F(\underline{x})\right] = \frac{1}{Z^{\underline{u},\underline{v}}_{M,+,0}} \mathbf{B}_M^{\underline{u},\underline{v}} \left[F(\underline{x})\mathbf{1}_{\Omega^n_M(\underline{x})}e^{-\mathcal{A}_M(\underline{x})}\right],
	\end{equation}
	where $F$ is any functional on $C([-M,M],\mathbb{R}^n)$, $\Omega^n_M(\underline{x}) = \{\underline{x}(t)\in\mathbb{A}_n^+ \mbox{ for all } t\in[-M,M]\}$, and 
	\[
	\mathcal{A}_M(\underline{x}) = \mathfrak{a}\sum_{i=1}^n \mathfrak{b}^{i-1}\int_{-M}^M x_i(t)\,dt.
	\]
	It is known \cite{CIW19, DLZ} by stochastic monotonicity that the zero boundary condition measures 
	\begin{equation}\label{zerobc}
	\mathbb{P}^{\underline{0},\underline{0}}_{M,+,0} := \lim_{\epsilon,\eta \, \downarrow \, 0} \mathbb{P}^{\epsilon\underline{x},\eta\underline{y}}_{M,+,0}
	\end{equation}
	exist and are independent of $\underline{x}$ and $\underline{y}$, and converge weakly (u.c.c.) as $M\to\infty$ to a measure $\mu_n$. Moreover, the free boundary condition measures $\mathbb{P}_{M,+,0}$ given by 
	\begin{equation}\label{freebpoly}
	\mathbb{P}_{M,+,0} \left[F(\underline{x})\right] = \frac{1}{Z_{M,+,0}}\int_{\mathbb{A}_n^+}\int_{\mathbb{A}_n^+} \mathbf{B}_M^{\underline{u},\underline{v}} \left[F(\underline{x})\mathbf{1}_{\Omega_M^n(\underline{x})} e^{-\mathcal{A}_M(\underline{x})}\right]d\underline{u}\,d\underline{v}
	\end{equation}
	are well-defined \cite[Appendix A]{CIW19}. It was proven in \cite[Theorem 1.1]{DLZ} that  these measures converge as $M\to\infty$ to the same limit $\mu_n$. Thus
	\begin{equation}\label{mun}
		\mu_n := \lim_{M\to\infty} \mathbb{P}^{\underline{0},\underline{0}}_{M,+,0} = \lim_{M\to\infty} \mathbb{P}_{M,+,0}, \quad \mbox{(u.c.c.)}
	\end{equation}
	We write $\mu_n^T$ for the restriction of $\mu_n$ to $[-T,T]$.
	
	In this paper we will also consider the similarly defined law $\mathbb{P}^{\underline{u}}_{M,+,0}$, namely that of $n$ Brownian motions on $[-M,M]$ starting at $\underline{u}$ at $-M$ with the same conditioning and area tilt. Because of the choice of normalization of $\mathbf{B}_M^{\underline{u},\underline{v}}$, this is equivalent to
	\begin{equation}\label{freemix}
	\mathbb{P}^{\underline{u}}_{M,+,0} \left[F(\underline{x})\right] = \frac{1}{Z^{\underline{u}}_{M,+,0}}\int_{\mathbb{A}_n^+} \mathbf{B}_M^{\underline{u},\underline{v}} \left[F(\underline{x})\mathbf{1}_{\Omega_M^n(\underline{x})} e^{-\mathcal{A}_M(\underline{x})}\right]d\underline{v}.
	\end{equation}
	It is not hard to see from the results of \cite{CIW2, DLZ} that $\mathbb{P}^{\underline{0}}_{M,+,0} := \lim_{\epsilon\downarrow 0}\mathbb{P}^{\epsilon\underline{x}}_{M,+,0}$ exists and converges as $M\to\infty$ to the same measure $\mu_n$; we will prove this in Section \ref{convpf}.
	
	Finally, we are interested in sending the number of curves $n$ to infinity. To emphasize that $n$ is no longer fixed, we will add a subscript of $n$ to the measures to indicate a growing number of curves. By \cite[Theorem 1.5]{CIW2}, the limiting law $\mu_n$ of the $n$-curve Brownian polymer itself has a weak limit (u.c.c.) as $n\to\infty$, which we denote by $\mu$. Thus
	\begin{equation}\label{mu}
		\mu := \lim_{n\to\infty}\lim_{M\to\infty} \mathbb{P}^{\underline{0},\underline{0}}_{n,M,+,0}, \quad \mbox{(u.c.c.)}
	\end{equation}
	In fact, the two limits can be taken in either order. For $T>0$ and $k\leq n$ fixed, write $\mathbb{P}^{\underline{u};T,k}_{n,a_N,+,\lambda_N}$, $\mathbb{P}^{\underline{u},\underline{v};T,k}_{n,a_N,+,\lambda_N}$, and $\mu^{T,k}$ for the corresponding laws restricted to $[-T,T]$ and the top $k$ curves.

	\subsection{Main results} Our first result proves that if $n$ is fixed and $N\to\infty$, then both the walk and bridge line ensemble measures converge after rescaling to the limiting law $\mu_n$ in \eqref{mun} of the Brownian polymer with $n$ curves.
	
	\begin{theorem}\label{main}
		Assume $V_\lambda(x) = \lambda x$. Let $\lambda_N$ be a sequence satisfying
		\[
		\lim_{N\to\infty} \lambda_N = 0, \qquad \lim_{N\to\infty} a_N := \lim_{N\to\infty} h_{\lambda_N}^2 N = \infty.
		\]
		Then for any $n\in\mathbb{N}$ and any bounded sequences $\underline{u}_N,\underline{v}_N\in\mathbb{A}_{n,\lambda_N}^+$, the measures $\mathbb{P}^{\underline{u}_N}_{a_N,+,\lambda_N}$ and $\mathbb{P}^{\underline{u}_N,\underline{v}_N}_{a_N,+,\lambda_N}$ both converge weakly $($u.c.c.$)$ to $\mu_n$ as $N\to\infty$. More precisely, for any $C\in(0,\infty)$ and $T>0$, uniformly in $u_{N,1},v_{N,1}\leq C$, we have the weak limits in the uniform topology
		\begin{align*}
			\lim_{N\to\infty}\mathbb{P}^{\underline{u}_N;T}_{a_N,+,\lambda_N}  = \lim_{N\to\infty}\mathbb{P}^{\underline{u}_N,\underline{v}_N;T}_{a_N,+,\lambda_N} = \mu_n^T.
		\end{align*}
	\end{theorem}

	As a corollary, we obtain convergence for $N\to\infty$ followed by the number of curves $n\to\infty$, with suitably chosen boundary conditions, to the limiting law $\mu$ in \eqref{mu} of the Brownian polymer with a growing number of curves.
	
	\begin{theorem}\label{infcurves}
		Assume the hypotheses of Theorem \ref{main}. Let $\underline{u}_N^n, \underline{v}_N^n \in \mathbb{A}_{n,\lambda_N}^+$ be sequences such that for each fixed $n\in\mathbb{N}$, $(\underline{u}_N^n)_{N\geq 1}$ and $(\underline{v}_N^n)_{N\geq 1}$ are bounded. Then the measures $\mathbb{P}^{\underline{u}_N^n}_{n,a_N,+,\lambda_N}$ and $\mathbb{P}^{\underline{u}_N^n,\underline{v}_N^n}_{n,a_N,+,\lambda_N}$ both converge weakly $($u.c.c.$)$ to $\mu$ if $N\to\infty$ first and then $n\to\infty$. That is, for any $T>0$ and $k\in\mathbb{N}$ we have the weak limits in the uniform topology
		\begin{align*}
			\lim_{n\to\infty}\lim_{N\to\infty}\mathbb{P}^{\underline{u}_N^n;T,k}_{n,a_N,+,\lambda_N} = \lim_{n\to\infty}\lim_{N\to\infty}\mathbb{P}^{\underline{u}_N^n,\underline{v}_N^n;T,k}_{n,a_N,+,\lambda_N} = \mu^{T,k}.
		\end{align*}
	\end{theorem}

	Together, Theorems \ref{main} and \ref{infcurves} address the open problem posed in \cite[3.5.3]{CIW2}. We note that there is no claim of uniformity with respect to boundary conditions in Theorem \ref{infcurves}. We make three remarks on these theorems before stating the final main result.
	
	\begin{remark}
	For concreteness, in Theorems \ref{main} and \ref{infcurves} one can take $\lambda_N = N^{-1}$. In this case the scaling in \eqref{rescaling} is given by $h_{\lambda_N} = N^{-1/3}$ in space and $H_{\lambda_N}^2 = N^{2/3}$ in time, i.e., diffusive 1:2:3 scaling. This agrees with the cube-root fluctuations described for SOS level lines in Section \ref{background}.
	\end{remark}

	\begin{remark}
	The order of limits taken in Theorem \ref{infcurves}, $N\to\infty$ followed by $n\to\infty$, is the same as in \cite[Theorem 1.1]{DLZ} for the Brownian polymer with free boundary conditions. On the other hand for the zero boundary condition Brownian polymer, \cite[Theorem 1.5]{CIW2} shows that the limits can be taken in any order. Their argument essentially amounts to showing that the top $k$ curves on $[-T,T]$ are stochastically increasing with $N$ and $n$. This is no longer true in our case since the walks are rescaled depending on $N$, unlike the Brownian polymer which has no rescaling. It would be interesting to prove a modification of Theorem \ref{infcurves} which allows $n$ to grow with $N$ at a sufficiently slow rate, but the mixing methods we use in this paper do not readily appear to accomplish this.
	\end{remark}
	
	\begin{remark}
		Although we only prove Theorems \ref{main} and \ref{infcurves} in the linear potential case, we expect them to hold more generally if one assumes, as in \cite{IVW}, that there is a nonnegative continuous function $q$ on $(0,\infty)$ such that $H_\lambda^2 V_\lambda(rH_\lambda) \to q(r)$ uniformly on compact sets as $\lambda\downarrow 0$. Our argument would apply in this case if we knew that the analogue of the Brownian polymer with the nonlinear area tilts $\exp(-\mathfrak{ab}^{i-1}\int_{-M}^M q(x_i(t))\,dt)$ converges as $M,n\to\infty$ to some analogue of $\mu$. This could likely be achieved by modifying the arguments in \cite{DLZ} given some mild assumptions on $q$, but we do not attempt to do so here as we believe the linear area case is the most relevant.
	\end{remark}

	The proof of Theorems \ref{main} and \ref{infcurves}, which we give in Section \ref{convpf}, will rely on the following mixing result, which is an analogue of \cite[Theorem 3.3]{IVW}.
	
		\begin{theorem}\label{mixing}
		For any $n\in\mathbb{N}$, $C\in(0,\infty)$, $T>0$, there exist $c_1,c_2>0$ such that for any $K>0$,
		\begin{align}
			\label{freemixing} \left\lVert \mathbb{P}^{\underline{r};T}_{a,+,\lambda} - \mathbb{P}^{\underline{w};T}_{b,+,\lambda}\right\rVert_{\mathrm{var}} \leq c_1 e^{-c_2 K},\\
			\label{bridgemixing} \left\lVert \mathbb{P}^{\underline{r},\underline{s};T}_{a,+,\lambda} - \mathbb{P}^{\underline{w},\underline{z};T}_{b,+,\lambda}\right\rVert_{\mathrm{var}} \leq c_1 e^{-c_2K},
		\end{align}
		uniformly in $\lambda$ small, $a,b\in\mathbb{Z}_\lambda$ with $a,b\geq K+T$, and $\underline{r},\underline{s},\underline{w},\underline{z}\in\mathbb{A}_{n,\lambda}^+$ with $r_1,s_1,w_1,z_1\leq C$.
	\end{theorem}
	The proof given in \cite{IVW} applies almost verbatim in the bridge case \eqref{bridgemixing}. Indeed, the only difference is that here the area tilt is geometric, i.e., $\mathfrak{b}>1$. In their argument, the area tilt is bounded from below using \eqref{q0}, with a constant factor of $n$ in front to account for each of the $n$ curves. This constant will increase to say $\mathfrak{b}^{n+1}$, but this does not affect the rest of the argument.
	
	In the walk case \eqref{freemixing}, the same argument works, but more modifications are needed. We will describe these in detail in Section \ref{mixpf}.
	
	\subsection*{Acknowledgments} The author would like to thank Amir Dembo for several helpful discussions regarding this work.

	\section{Proof of convergence}\label{convpf}
	
	In this section we prove Theorems \ref{main} and \ref{infcurves}. 
	
	\subsection{Preliminaries} We will need two short lemmas for the proof. The first can be viewed as an invariance principle for the random walk ensembles to the Brownian polymers. It gives convergence to the Brownian polymer if the time scale is fixed and the mesh size tends to zero. In combination with the mixing statement Theorem \ref{mixing} and the convergence results for the Brownian polymer, this will quickly imply the main results. 
	
	\begin{lemma}\label{Mconv}
		Assume $V_\lambda(x) = \lambda x$. Let $\underline{u}_ N,\underline{v}_N$ be sequences in $\mathbb{A}_{n,\lambda_N}^+$ such that $\underline{u}_N \to \underline{u}\in\mathbb{A}_n^+$ and $\underline{v}_N\to\underline{v}\in\mathbb{A}_n^+$ as $N\to\infty$. Fix $M>0$ and write $M_N = \lambda_N^{2/3} \lceil \lambda_N^{-2/3} M\rceil$, so that $M_N\in\mathbb{Z}_{\lambda_N}$ and $M_N\downarrow M$. Fix $T\leq M$, and let $F$ be any continuous bounded functional on $C([-T,T],\mathbb{R}^n)$. Then 
		\begin{align}
			\label{walkM} \lim_{N\to\infty}\mathbb{P}^{\underline{u}_N}_{M_N,+,\lambda_N} \big[F(\underline{x}^{\lambda_N}|_{[-T,T]})\big] = \mathbb{P}^{\underline{u}}_{M,+,0} \big[F(\underline{x}|_{[-T,T]})\big],\\
			\label{bridgeM} \lim_{N\to\infty}\mathbb{P}^{\underline{u}_N,\underline{v}_N}_{M_N,+,\lambda_N} \big[F(\underline{x}^{\lambda_N}|_{[-T,T]})\big] = \mathbb{P}^{\underline{u},\underline{v}}_{M,+,0} \big[F(\underline{x}|_{[-T,T]})\big].
		\end{align}
	\end{lemma}
	
	\begin{proof}
		We write $\lambda$ in place of $\lambda_N$ for brevity; it will be clear from context which index $N$ we take. We mostly work in the bridge case \eqref{bridgeM} and explain the adjustments needed for the walk case \eqref{walkM}.
		
		Let us write $\underline{s}^{\lambda}$ for $n$ independent random walk bridges distributed according to $\mathbb{P}^{H_{\lambda}\underline{u}_N, H_{\lambda}\underline{v}_N}_{H_{\lambda}^2M_N}$, rescaled as in \eqref{rescaling}. Because of the Brownian scaling, it is known by an invariance principle for bridges (see, e.g., \cite[Theorem 4]{Lig}) that the law of $\underline{s}^{\lambda}|_{[-M,M]}$ converges as $N\to\infty$ to the law of $n$ independent Brownian bridges $\underline{B}$ on $[-M,M]$ from $\underline{u}$ to $\underline{v}$. Since $C([-M,M],\mathbb{R}^n)$ with the uniform topology is separable, by the Skorohod representation theorem there is a probability measure $\mathbb{P}$ on some probability space supporting $C([-M,M],\mathbb{R}^n)$-valued random variables $\underline{y}^{\lambda}$ and $\underline{y}$ with the laws of $\underline{s}^{\lambda}|_{[-M,M]}$ and $\underline{B}$ respectively, such that $\underline{y}^{\lambda} \to \underline{y}$ uniformly on $[-M,M]$ as $N\to\infty$, $\mathbb{P}$-a.s.
		
		Consider first the indicators of the curves remaining ordered in \eqref{bridge} and \eqref{bpoly}. For $a<b$, let us write
		\begin{equation}\label{omega}
			\Omega^{n,+}_{a,b} := \{\underline{f}\in C([a,b],\mathbb{R}^n) : \underline{f}(t)\in\mathbb{A}_n^+ \mbox{ for all } t\in[a,b]\}.
		\end{equation}
		If $a=-b$, we write instead $\Omega^{n,+}_b$. For $\underline{s}^{\lambda}$, we can express the indicator in the definition \eqref{bridge} (if $\underline{s}^{\lambda}$ is taken to be the rescaling and linear interpolation of the bridge $\underline{X}$) as the indicator that $\underline{s}^{\lambda}$ lies in the set $\Omega^{n,+}_M$. Indeed, since $\underline{s}^{\lambda}(-M_N) = \underline{u}_N$ and $\underline{s}^{\lambda}(M_N) = \underline{v}_N$ already lie in $\mathbb{A}_n^+$, and $M_N$ is the smallest element of $\mathbb{Z}_{\lambda}$ larger than $M$, the indicator in \eqref{bridge} is exactly equal to $\mathbf{1}_{\Omega_M^{n,+}}(\underline{s}^{\lambda})$, which is equal in law to $\mathbf{1}_{\Omega_M^{n,+}}(\underline{y}^{\lambda})$ under $\mathbb{P}$. On the other hand, the indicator appearing in \eqref{bpoly} for $\underline{B}$ is simply $\mathbf{1}_{\Omega_M^{n,+}}(\underline{B})$, which is equal in law to $\mathbf{1}_{\Omega_M^{n,+}}(\underline{y})$ under $\mathbb{P}$.
		
		Now observe that $\Omega_M^{n,+}$ is an open subset of $C([-M,M],\mathbb{R}^n)$. Indeed, suppose $\underline{f}\in\Omega^{n,+}_M$. Since $\underline{f}$ is continuous and $[-M,M]$ is compact, there exists $\epsilon>0$ so that $\min_{1\leq i <n, t\in[-M,M]} (f_i(t)-f_{i+1}(t)) > \epsilon$. Then clearly the $\epsilon/2$-ball around $\underline{f}$ in the sup-norm is still contained in $\Omega^{n,+}_M$. It follows that the indicator $\mathbf{1}_{\Omega^{n,+}_M}$ is lower semicontinuous, which implies
		\[
		\liminf_{N\to\infty} \mathbf{1}_{\Omega^{n,+}_M}(\underline{y}^{\lambda}) \geq \mathbf{1}_{\Omega^{n,+}_M}(\underline{y}), \quad \mathbb{P}\mbox{-a.s.}
		\]

		For the other bound, consider instead the closed Weyl chamber $\mathbb{A}_n^0$ and the corresponding set ${\Omega}^{n,0}_M = \{\underline{f}\in C([-M,M],\mathbb{R}^n) : \underline{f}(t) \in\mathbb{A}_n^0 \mbox{ for all } t\in[-M,M]\}$. It is easy to see that ${\Omega}^{n,0}_M$ is closed, so $\mathbf{1}_{{\Omega}^{n,0}_M}$ is upper semicontinuous. Therefore
		\[
		\limsup_{N\to\infty} \mathbf{1}_{{\Omega}^{n,0}_M}(\underline{y}^{\lambda}) \leq \mathbf{1}_{{\Omega}^{n,0}_M}(\underline{y}), \quad \mathbb{P}\mbox{-a.s.}
		\]
		Now the key point is that $\mathbf{1}_{\Omega^{n,+}_M}(\underline{y}) = \mathbf{1}_{{\Omega}^{n,0}_M}(\underline{y})$, $\mathbb{P}$-a.s. Indeed, the complement of these two events for $\underline{y}$ is the event that $\min_{1\leq i<n,t\in[-M,M]} (y_i(t) - y_{i+1}(t)) = 0$. This has probability 0 since the difference of two Brownian bridges is another Brownian bridge, and the minimum of a Brownian bridge is a continuous random variable by the reflection principle. Therefore, combining the above two inequalities and using the trivial inclusion $\Omega^{n,+}_M \subset {\Omega}^{n,0}_M$, we get
		\[
		\limsup_{N\to\infty} \mathbf{1}_{\Omega^{n,+}_M}(\underline{y}^{\lambda}) \leq \limsup_{N\to\infty} \mathbf{1}_{{\Omega}^{n,0}_M}(\underline{y}^{\lambda}) \leq \mathbf{1}_{{\Omega}^{n,0}_M}(\underline{y}) = \mathbf{1}_{\Omega^{n,+}_M}(\underline{y}) \leq \liminf_{N\to\infty} \mathbf{1}_{\Omega^{n,+}_M}(\underline{y}^{\lambda}), \quad \mathbb{P}\mbox{-a.s.}
		\]
		This implies that 
		\begin{equation}\label{indiclim}
			\lim_{N\to\infty} \mathbf{1}_{\Omega^{n,+}_M}(\underline{y}^{\lambda}) = \mathbf{1}_{\Omega^{n,+}_M}(\underline{y}), \quad \mathbb{P}\mbox{-a.s.}
		\end{equation}
		
		For the walk case, we must be slightly more careful at the right endpoint $M$. Write $\underline{r}^{\lambda_N}$ for $n$ independent random walks distributed according to $\mathbb{P}^{H_{\lambda}\underline{u}_N}_{-H_{\lambda}^2 M_N}$ and rescaled as in \eqref{rescaling}. By the invariance principle the law of $\underline{r}^{\lambda}|_{[-M,M+1]}$ converges as $N\to\infty$ to that of $n$ independent Brownian motions on $[-M,M+1]$ starting at $\underline{u}$. Again by the Skorohod representation theorem we can find a probability measure $\mathbb{P}$ and $C([-M,M+1],\mathbb{R}^n)$-valued random variables $\underline{z}^{\lambda}$ and $\underline{z}$ with the same respective laws, so that $\underline{z}^{\lambda}\to\underline{z}$ uniformly, $\mathbb{P}$-a.s. Then the indicator appearing in \eqref{walk} has the same law as $\mathbf{1}_{\Omega^{n,+}_{-M,M_N}}(\underline{z}^{\lambda}|_{[-M,M_N]})$, with notation as in \eqref{omega}. Now since $M_N\downarrow M$, for any $\delta>0$ we have for sufficiently large $N$ that
		\[
		\mathbf{1}_{\Omega_{-M,M+\delta}^{n,+}}(\underline{z}^{\lambda}|_{[-M,M+\delta]}) \leq \mathbf{1}_{\Omega^{n,+}_{-M,M_N}}(\underline{z}^{\lambda}|_{[-M,M_N]}) \leq \mathbf{1}_{\Omega_{-M,M}^{n,+}}(\underline{z}^{\lambda}|_{[-M,M]}).
		\]
		By the argument leading up to \eqref{indiclim}, the left and right hand sides converge as $N\to\infty$, $\mathbb{P}$-a.s., to $\mathbf{1}_{\Omega_{-M,M+\delta}^{n,+}}(\underline{z}|_{[-M,M+\delta]})$ and $\mathbf{1}_{\Omega_{-M,M}^{n,+}}(\underline{z}|_{[-M,M]})$. Now sending $\delta\downarrow 0$, it is clear since $\underline{z}$ is continuous that the first indicator tends $\mathbb{P}$-a.s. to the latter. Therefore
		\begin{equation}\label{walkindic}
			\lim_{N\to\infty} \mathbf{1}_{\Omega_{-M,M_N}^{n,+}}(\underline{z}^{\lambda_N}|_{[-M,M_N]}) = \mathbf{1}_{\Omega_M^{n,+}}(\underline{z}|_{[-M,M]}), \quad \mathbb{P}\mbox{-a.s.}
		\end{equation}
		
		Next consider the area tilts. We work with bridges; there is no change for walks. A small amount of care is needed since $\underline{s}^{\lambda}$ is defined on the larger interval $[-M_N,M_N]$, while $\underline{y}^{\lambda}$ is only defined on $[-M,M]$. In the linear case $V_\lambda(x) = \lambda x$, using \eqref{arearescaled} we write the area tilt for $s_i^{\lambda}$ as
		\[
		\sum_{k=-M_N\lambda^{-2/3}}^{M_N\lambda^{-2/3}-1} \lambda \cdot\lambda^{-1/3}s_i^{\lambda}(\lambda^{2/3}k) = \lambda^{2/3} u_{N,i} +  \lambda^{2/3}\sum_{-M_N\lambda^{-2/3}+1}^{M_N\lambda^{-2/3}-1} s_i^{\lambda}(\lambda^{2/3}k).
		\] 
		Since $M_N\lambda_N^{-2/3} - 1 \leq M$ by construction, and the law of $y_i^{\lambda}$ under $\mathbb{P}$ is the law of $s_i^{\lambda}|_{[-M,M]}$, it now makes sense to define
		\[
		\mathcal{A}_M^{\lambda}(\underline{y}^{\lambda}) := \mathfrak{a}\sum_{i=1}^n \mathfrak{b}^{i-1} \left[ \lambda^{2/3} u_{N,i} +  \lambda^{2/3}\sum_{k=-M_N\lambda^{-2/3}+1}^{M_N\lambda^{-2/3}-1} y_i^{\lambda}(\lambda^{2/3}k) \right].
		\]
		Then the law of $\mathcal{A}_M^{\lambda}(\underline{y}^{\lambda})$ under $\mathbb{P}$ is exactly the law of $\mathcal{A}_{M_N}^{\lambda}(\underline{s}^{\lambda})$. Note that the first term in the sum $\lambda_N^{2/3}u_{N,i}\to 0$ since $u_{N,i}\to u_i < \infty$. Now we estimate
		\begin{align*}
			\left|\lambda^{2/3}\sum_{-M_N\lambda^{-2/3}+1}^{M_N\lambda^{-2/3}-1} y_i^{\lambda}(\lambda^{2/3}k) - \int_{-M}^M y_i(t)\,dt\right| &\leq \lambda^{2/3}\sum_{-M_N\lambda^{-2/3}+1}^{M_N\lambda^{-2/3}-1} \left|y_i^{\lambda}(\lambda^{2/3}k) - y_i(\lambda^{2/3}k)\right|\\
			&\qquad + \left|\lambda^{2/3}\sum_{-M_N\lambda^{-2/3}+1}^{M_N\lambda^{-2/3}-1} y_i(\lambda^{2/3}k) - \int_{-M}^M y_i(t)\,dt\right|\\
			&\leq 2M_N\lVert y_i^{\lambda} - y_i\rVert_\infty + o_N(1)\\ 
			&= o_N(1),
		\end{align*}
		where the term in the second line is seen to be $o_N(1)$ since the sum is a Riemann sum for the continuous function $y_i$. This argument applies to the walks $\underline{z}^\lambda$ as well, so we have
		\begin{equation}\label{arealim}
			\mathcal{A}_M^{\lambda}(\underline{y}^{\lambda}) \longrightarrow \mathcal{A}_M(\underline{y}), \quad \mathcal{A}_M^{\lambda}(\underline{z}^{\lambda}) \longrightarrow \mathcal{A}_M(\underline{z}), \quad \mathbb{P}\mbox{-a.s.}
		\end{equation}
		
		We conclude using the bounded convergence theorem. For bridges the expectations can be written as
		\begin{align*}
			\mathbb{P}^{\underline{u}_N,\underline{v}_N}_{M_N,+,\lambda_N} \big[F(\underline{x}^{\lambda}|_{[-T,T]})\big] &= \frac{ \mathbb{P} \big[F(\underline{y}^{\lambda}|_{[-T,T]})\mathbf{1}_{\Omega^{n,+}_M}(\underline{y}^{\lambda}) e^{-\mathcal{A}^{\lambda}_M(\underline{y}^{\lambda})}\big] } {\mathbb{P} \big[\mathbf{1}_{\Omega^{n,+}_M}(\underline{y}^{\lambda}) e^{-\mathcal{A}^{\lambda}_M(\underline{y}^{\lambda})}\big] }, \\
			\mathbb{P}^{\underline{u},\underline{v}}_{M,+,0} \big[F(\underline{x}|_{[-T,T]})\big] &= \frac{ \mathbb{P} \big[F(\underline{y}|_{[-T,T]})\mathbf{1}_{\Omega^{n,+}_M}(\underline{y})e^{-\mathcal{A}_M(\underline{y})}\big] }{ \mathbb{P} \big[\mathbf{1}_{\Omega^{n,+}_M}(\underline{y})e^{-\mathcal{A}_M(\underline{y})}\big]} .
		\end{align*}
		Since $F$ is continuous and $\underline{y}^{\lambda}\to\underline{y}$ uniformly on $[-M,M] \supseteq [-T,T]$, we have that $F(\underline{y}^{\lambda}|_{[-T,T]}) \to F(\underline{y}|_{[-T,T]})$, $\mathbb{P}$-a.s. Combining with \eqref{indiclim} and \eqref{arealim}, we see that the integrands on the right in the first line, in both numerator and denominator, converge $\mathbb{P}$-a.s.\!\! to the those in the second line as $N\to\infty$. Since the integrands are uniformly bounded by $\sup |F|$ in the numerator and 1 in the denominator, \eqref{bridgeM} follows. Likewise, \eqref{walkindic} and \eqref{arealim} imply \eqref{walkM}.
	\end{proof}
	
	The next lemma will be needed to treat the walk case in Theorem \ref{main}.
	
	\begin{lemma}\label{freeinv}
		The measures
		\[
		\mathbb{P}^{\underline{0}}_{M,+,0} := \lim_{\epsilon\downarrow 0} \mathbb{P}^{\epsilon\underline{w}}_{M,+,0}
		\]
		exist and are independent of $\underline{w}\in\mathbb{A}_n^+$. Moreover, $\mathbb{P}^{\underline{0}}_{M,+,0}$ converges weakly $($u.c.c.$)$ to $\mu_n$ as $M\to\infty$.
	\end{lemma}
	
	\begin{proof}
		We will prove both statements using the stochastic monotonicity result \cite[Lemma 1.2]{CIW19}. For the first statement, note that the same coupling argument as in \cite[Appendix B]{CIW19}, except with the constraint on the right boundary removed, shows that 
		\begin{equation}\label{monot}
			\underline{u} \leq \underline{v} \quad \mathrm{implies} \quad \mathbb{P}^{\underline{u}}_{M,+,0} \preceq \mathbb{P}^{\underline{v}}_{M,+,0}.
		\end{equation}
		Here $\underline{u}\leq\underline{v}$ means $u_i \leq v_i$ for $1\leq i\leq n$, and $\preceq$ denotes stochastic ordering of measures. (We recall that for two probability measures $\mu,\nu$ the stochastic ordering $\mu \preceq \nu$ means that for any increasing functional $F$ we have $\mu[F(X)] \leq \nu[F(X)]$.) The existence statement of the lemma now follows by monotone convergence.
		
		To prove the second statement, we will show that 
		\begin{equation}\label{squeeze}
			\mathbb{P}^{\underline{0},\underline{0}}_{M,+,0} \preceq \mathbb{P}^{\underline{0}}_{M,+,0} \preceq \mathbb{P}_{M,+,0},
		\end{equation} 
		where we recall $\mathbb{P}_{M,+,0}$ denotes free boundary conditions as in \eqref{freebpoly}. Since the left and right measures are both known to converge to $\mu_n$ by \eqref{mun}, this implies the result. Let $F$ be an increasing functional on $C([-M,M],\mathbb{A}_n^+)$. To prove the first inequality in \eqref{squeeze}, we note by \eqref{freemix} that $\mathbb{P}^{\epsilon\underline{w}}_{M,+,0}$ can be written as a mixture of the bridge measures via
		\[
		\mathbb{P}^{\epsilon\underline{w}}_{M,+,0}[F(\underline{x})] = \int_{\mathbb{A}_n^+} \mathbb{P}^{\epsilon\underline{w},\underline{v}}_{M,+,0}[F(\underline{x})]  \frac{Z^{\epsilon\underline{w},\underline{v}}_{M,+,0}}{Z^{\epsilon\underline{w}}_{M,+,0}}\,d\underline{v}.
		\]
		Since $\mathbb{P}^{\epsilon\underline{w},\underline{v}}_{M,+,0} \succeq \lim_{\eta\downarrow 0} \mathbb{P}^{\epsilon\underline{w},\eta\underline{v}}_{M,+,0} =  \mathbb{P}^{\epsilon\underline{w},\underline{0}}_{M,+,0}$ for each $\underline{v}$, we get
		\[
		\mathbb{P}^{\epsilon\underline{w}}_{M,+,0}[F(\underline{x})]  \geq 		\mathbb{P}^{\epsilon\underline{w},\underline{0}}_{M,+,0}[F(\underline{x})]  \int_{\mathbb{A}_n^+}  \frac{Z^{\epsilon\underline{w},\underline{v}}_{M,+,0}}{Z^{\epsilon\underline{w}}_{M,+,0}}\,d\underline{v} = \mathbb{P}^{\epsilon\underline{w},\underline{0}}_{M,+,0}[F(\underline{x})].
		\]
		Taking $\epsilon\downarrow 0$ implies the first inequality in \eqref{squeeze}. Similarly, for the second inequality \eqref{monot} implies
		\[
		\mathbb{P}_{M,+,0}[F(\underline{x})]  = \int_{\mathbb{A}_n^+} \mathbb{P}^{\underline{u}}_{M,+,0}[F(\underline{x})] 	\frac{Z^{\underline{u}}_{M,+,0}}{Z_{M,+,0}}\,d\underline{u} \geq \mathbb{P}^{\underline{0}}_{M,+,0}[F(\underline{x})]  \int_{\mathbb{A}_n^+} \frac{Z^{\underline{u}}_{M,+,0}}{Z_{M,+,0}}\,d\underline{u} = \mathbb{P}^{\underline{0}}_{M,+,0}[F(\underline{x})].
		\]
	\end{proof}

	\subsection{Proof of Theorems \ref{main} and \ref{infcurves}}
	
	We now complete the proof of convergence.
	
	\begin{proof}
		First consider bridges. Let $F$ be any bounded continuous functional on $C([-T,T],\mathbb{R}^n)$. Fix $M>T$ and define $M_N$ as in the statement of Lemma \ref{Mconv}. By the mixing bound \eqref{bridgemixing} in Theorem \ref{mixing}, we have
		\[
		\mathbb{P}_{a_N,+,\lambda_N}^{\underline{u}_N,\underline{v}_N} \big[F(\underline{x}^{\lambda_N}|_{[-T,T]})\big] = \mathbb{P}_{M_N,+,\lambda_N}^{\underline{r},\underline{s}} \big[F(\underline{x}^{\lambda_N}|_{[-T,T]})\big] + R_M
		\]
		where $R_M$ is an error term satisfying $|R_M|\leq c_1 e^{-c_2(M-T)}$, uniformly in $u_{N,1},v_{N,1},r_1,s_1\leq C$ and large $N$. In particular, we can fix $\epsilon>0$ and $\underline{w}\in\mathbb{A}_n^+$ and choose $\underline{r}=\underline{s} = \underline{w}_N$ for some sequence $\underline{w}_N\to \epsilon \underline{w}$. Lemma \ref{Mconv} then implies that
		\begin{align*}
			\limsup_{N\to\infty} \mathbb{P}_{a_N,+,\lambda_N}^{\underline{u}_N,\underline{v}_N} \big[F(\underline{x}^{\lambda_N}|_{[-T,T]})\big]  &\leq \mathbb{P}^{\epsilon\underline{w},\epsilon\underline{w}}_{M,+,0} \big[ F(\underline{x}|_{[-T,T]})\big] + c_1e^{-c_2(M-T)},\\
			\liminf_{N\to\infty} \mathbb{P}_{a_N,+,\lambda_N}^{\underline{u}_N,\underline{v}_N} \big[F(\underline{x}^{\lambda_N}|_{[-T,T]})\big]  &\geq \mathbb{P}^{\epsilon\underline{w},\epsilon\underline{w}}_{M,+,0} \big[ F(\underline{x}|_{[-T,T]})\big] - c_1e^{-c_2(M-T)}.
		\end{align*}
		Now first taking $\epsilon\downarrow 0$, and then $M\to\infty$, by \eqref{mun} the right hand sides of the above two inequalities both converge to the same limit. Combining, we see that the limit exists and
		\begin{equation}\label{thm1pf}
			\lim_{N\to\infty} \mathbb{P}_{a_N,+,\lambda_N}^{\underline{u}_N,\underline{v}_N} \big[F(\underline{x}^{\lambda_N}|_{[-T,T]})\big] = \mu_n \big[F(\underline{x}|_{[-T,T]})\big].
		\end{equation}
		For the walk case, we instead apply \eqref{freemixing} in Theorem \ref{mixing} to get
		\[
		\mathbb{P}_{a_N,+,\lambda_N}^{\underline{u}_N} \big[F(\underline{x}^{\lambda_N}|_{[-T,T]})\big] = \mathbb{P}_{M_N,+,\lambda_N}^{\underline{r}} \big[F(\underline{x}^{\lambda_N}|_{[-T,T]})\big] + R_M,
		\] 
		and then in view of Lemma \ref{freeinv} the same argument shows that
		\begin{equation}\label{thm1pf2}
			\lim_{N\to\infty} \mathbb{P}_{a_N,+,\lambda_N}^{\underline{u}_N} \big[F(\underline{x}^{\lambda_N}|_{[-T,T]})\big] = \mu_n\big[F(\underline{x}|_{[-T,T]})\big].
		\end{equation}
		Since $F$ was arbitrary, \eqref{thm1pf} and \eqref{thm1pf2} prove Theorem \ref{main}.
		
		Now to prove Theorem \ref{infcurves}, fix any $k\leq n$ and any bounded continuous functional $F$ on $C([-T,T],\mathbb{R}^k)$. Let $\pi_k$ denote the projection of $\mathbb{R}^n$ onto the first $k$ coordinates in $\mathbb{R}^k$. If $\underline{u}_N^n,\underline{v}_N^n \in \mathbb{A}_{n,\lambda_N}^+$ remain bounded in $N$, then \eqref{thm1pf} and \eqref{thm1pf2} imply that
		\begin{align*}
			\lim_{N\to\infty} \mathbb{P}_{n,a_N,+,\lambda_N}^{\underline{u}_N^n} \big[F(\pi_k\circ\underline{x}^{\lambda_N}|_{[-T,T]})\big] = \mu_n \big[F(\pi_k\circ\underline{x}|_{[-T,T]})\big],\\
			\lim_{N\to\infty} \mathbb{P}_{n,a_N,+,\lambda_N}^{\underline{u}_N^n,\underline{v}_N^n} \big[F(\pi_k\circ\underline{x}^{\lambda_N}|_{[-T,T]})\big] = \mu_n\big[F(\pi_k\circ\underline{x}|_{[-T,T]})\big],
		\end{align*} 
		for each $n\geq 1$. By \eqref{mu}, the right hand side of both lines converges to $\mu[F(\pi_k\circ\underline{x}|_{[-T,T]})]$ as $n\to\infty$. Again since $F$ and $k$ are arbitrary, this proves Theorem \ref{infcurves}.
	\end{proof}

\section{Mixing for walks}\label{mixpf}

In this section we explain how the mixing bounds \eqref{freemixing} in Theorem \ref{mixing} can be proven in the case of random walks (that is, fixed only on the left) with geometric area tilts. We emphasize that we follow very closely the argument of \cite{IVW}, making only small modifications where needed to account for the different boundary conditions. For convenience, we will assume without loss of generality that $H_\lambda^2 \in \mathbb{Z}$ in this section, so that $\mathbb{Z}\subset \mathbb{Z}_\lambda$.

\subsection{Reduction to good blocks}

As in \cite[Section 6]{IVW}, we fix $\eta>0$ large and $\epsilon>0$ small, and define the regular sets $\mathbb{A}_n^{+,\mathsf{r}} = \{\underline{x}\in\mathbb{A}_n^+ : x_1 \leq \eta \mbox{ and } \min_{1\leq i < n} (x_i-x_{i+1}) \geq \epsilon\}$ and $\mathbb{A}_{n,\lambda}^{+,\mathsf{r}} = \mathbb{A}_{n,\lambda}^+ \cap \mathbb{A}_n^{+,\mathsf{r}}$. An interval $[\ell,\ell+1]$ is called regular for a trajectory $\underline{x}(\cdot)$ if $\underline{x}(\ell), \underline{x}(\ell+1) \in \mathbb{A}_n^{+,\mathsf{r}}$ and $\max_{t\in[\ell,\ell+1]} x_1(t) \leq 2\eta$. For a block $D_\ell = D_\ell^- \cup D_\ell^+ := [2\ell,2\ell+1] \cup [2\ell+1,2(\ell+1)]$, we say $D_\ell$ is good if both $D_\ell^+$ and $D_\ell^-$ are regular. We also write $D_\ell$ for the event that the block $D_\ell$ is good for the trajectory $\underline{x}^\lambda$.

Consider a couple of independent trajectories $(\underline{x}^\lambda, \underline{y}^\lambda)$ distributed according to $\mathbb{P}^{\underline{r}}_{a,+,\lambda} \otimes \mathbb{P}^{\underline{u}}_{b,+,\lambda}$. We let $3M = a\wedge b$, and for $D_\ell \subset [-2M,2M]$ we write $\mathfrak{D}_\ell^\pm$ for the event that $D_\ell^\pm$ is good for both $\underline{x}^\lambda$ and $\underline{y}^\lambda$, and $\mathfrak{D}_\ell = \mathfrak{D}_\ell^+ \cap \mathfrak{D}_\ell^-$. We define 
\[
\mathcal{M}_0 = \sum_{-M\leq\ell\leq M-1} \mathbf{1}_{\mathfrak{D}_\ell}.
\]
We will prove below the following analogue of \cite[Lemma 6.2]{IVW}.

\begin{lemma}\label{good}
	For $\eta>0$ large enough and $\epsilon>0$ small enough, there exist $\nu,\kappa>0$ such that
	\[
	\mathbb{P}^{\underline{r}}_{a,+,\lambda} \otimes \mathbb{P}^{\underline{u}}_{b,+,\lambda} (\mathcal{M}_0 \leq \nu M) \leq e^{-\kappa M},
	\]
	uniformly in $\lambda$ small, $M$ large, and $r_1,u_1\leq \eta$.
\end{lemma}

Given this lemma, the proof of the bounds \eqref{freemixing} in Theorem \ref{mixing} proceeds in exactly the same way as in \cite[Section 6.4]{IVW}. The argument there relies on the Gibbs property for the bridge measures, but the walk measures satisfy precisely the same Gibbs property (see Remark \ref{gibbs}), and thus the argument translates immediately. It therefore remains to prove Lemma \ref{good}.

\subsection{Proof of Lemma \ref{good}}
As in \cite[Section 7.3]{IVW}, we define the notion of a pre-good 5-block as follows. For integers $-\lfloor M/5\rfloor \leq \ell \leq \lfloor M/5\rfloor$, a 5-block $D_\ell^{(5)} = D_{5\ell-2} \cup \cdots \cup D_{5\ell+2}$ is called pre-good for $\underline{x}^\lambda$ if 
\[
\min_{t\in D_{5\ell-2}} x_1^\lambda(t) \leq \eta, \quad \min_{t\in D_{5\ell+2}} x_1^\lambda(t) \leq \eta.
\]
Let $\tilde{\mathfrak{D}}_{5\ell}$ denote the event that $D_\ell^{(5)}$ is jointly pre-good for $\underline{x}^\lambda$ and $\underline{y}^\lambda$. By \cite[(7.10)]{IVW}, we have
\[
\mathbb{P}^{\underline{r},\underline{s}}_{-4,6,+,\lambda}(D_0 \mbox{ is good} \mid D_0^{(5)} \mbox{ is pre-good}) \geq \rho_1
\]
for a constant $\rho_1(\eta,\epsilon)$ uniformly in $\underline{r},\underline{s}\in\mathbb{A}_{n,\lambda}^+$ and small $\lambda$. (This is proven for $\mathfrak{b}=1$, but in our case $\mathfrak{b}>1$ just results in a different constant.) By the Gibbs property, this implies that for each $\ell$,
\[
\mathbb{P}^{\underline{r}}_{a,+,\lambda}\otimes\mathbb{P}^{\underline{s}}_{b,+,\lambda} (\mathfrak{D}_{5\ell} \mid \tilde{\mathfrak{D}}_{\ell}^{(5)}) \geq \rho_1^2.
\]
It therefore suffices to prove that with
\[
\mathcal{M}_0^{(5)} = \sum_{j=-\lfloor M/5\rfloor}^{\lfloor M/5\rfloor} \mathbf{1}_{\mathfrak{D}_\ell^{(5)}},
\]
we can find constants $\nu^{(5)}$ and $\kappa^{(5)}$ such that
\begin{equation}\label{pregood}
\mathbb{P}^{\underline{r}}_{a,+,\lambda}\otimes\mathbb{P}^{\underline{s}}_{b,+,\lambda} (\mathcal{M}_0^{(5)} \leq \nu^{(5)}M) \leq e^{-\kappa^{(5)}M}
\end{equation}
uniformly in small $\lambda$, large $M$, $a,b\geq 3M$, and $r_1,s_1\leq \eta$.

To prove \eqref{pregood}, we will need the following analogue of \cite[Lemma 7.1]{IVW}.

\begin{lemma}\label{partition}
	There exist constants $c_1(n)$, $c_2(n,\eta)$, and $T_0(\eta)$ such that for all $T\geq T_0$,
	\[
	Z^{\underline{w}}_{T,+,\lambda} \geq c_2 e^{-c_1T}
	\] 
	uniformly in $w_1\leq \eta$ and $\lambda$ small.
\end{lemma}

\begin{proof}
	Fix $\epsilon>0$ so that $n\epsilon < 1$, and assume $T>2$. We write $\hat{\mathbf{P}}^{\underline{w}}_{\lambda}$ for the law of $n$ independent rescaled random walks started at $\underline{w}$, and $\hat{\mathbf{P}}^{\underline{w}}_{t,+,\lambda}$ for this law with trajectories restricted to stay in $\mathbb{A}_{n,\lambda}^+$ in the interval $[0,t]$. Write $\mathbb{A}_{n,\lambda}^{+,\mathsf{r}}(\alpha) = \mathbb{A}_{n,\lambda}^{+,\mathsf{r}}\cap\{\underline{x} : x_1\leq\alpha\}$, and define the event
	\begin{align*}
		\mathcal{E} = \left\{\max_{t\in[0,1]} x_1^\lambda(t) \leq 2\eta, \, \underline{x}^\lambda(1) \in \mathbb{A}_{n,\lambda}^{+,\mathsf{r}}(1),\max_{t\in[1,2T]} x_1^\lambda(t) < 2\right\}.
	\end{align*}
	By \eqref{q0}, we have
	\begin{equation}\label{Zbd1}
	Z^{\underline{w}}_{T,+,\lambda} \geq e^{-\mathfrak{ab}^{n+1} (q_0(2\eta) + (2T-1)q_0(2))} \hat{\mathbf{P}}^{\underline{w}}_{2T,+,\lambda} (\mathcal{E}).
	\end{equation}
	To bound this probability, we use the Markov property to compute
	\begin{align*}
		\hat{\mathbf{P}}^{\underline{w}}_{2T,+,\lambda} (\mathcal{E}) 
		&= \sum_{\underline{u}\in\mathbb{A}_{n,\lambda}^{+,\mathsf{r}}(1)} \hat{\mathbf{P}}^{\underline{w}}_{2T,+,\lambda} \left(\max_{t\in[0,1]} x_1^\lambda(t) \leq 2\eta, \max_{t\in[1,2T]} x_1^\lambda(t) < 2 \, \bigg| \, \underline{x}^\lambda(1) = \underline{u}\right)\\
		&\quad\quad\quad\quad\quad \quad \times \hat{\mathbf{P}}^{\underline{w}}_{2T,+,\lambda} (\underline{x}^\lambda(1) = \underline{u}) \\
		&= \sum_{\underline{u}\in \mathbb{A}_{n,\lambda}^{+,\mathsf{r}}(1)} \hat{\mathbf{P}}^{\underline{w}}_{1,+,\lambda} \left(\max_{t\in[0,1]} x_1^\lambda(t) \leq 2\eta\, \bigg| \, \underline{x}^\lambda(1) = \underline{u}\right) \hat{\mathbf{P}}^{\underline{u}}_{2T-1,+,\lambda}\left(\max_{t\in[0,2T-1]} x_1^\lambda(t) < 2\right)\\
		&\quad\quad\quad\quad\quad \quad \times \hat{\mathbf{P}}^{\underline{w}}_{1,+,\lambda} (\underline{x}^\lambda(1) = \underline{u}) \\
		&\geq \hat{\mathbf{P}}^{\underline{w}}_{1,+,\lambda} \left(\max_{t\in[0,1]} x_1^\lambda(t)\leq 2\eta, \, \underline{x}^\lambda(1)\in\mathbb{A}_{n,\lambda}^{+,\mathsf{r}}(1)\right) \\
		&\qquad\qquad\qquad\times \min_{u\in\mathbb{A}_{n,\lambda}^{+,\mathsf{r}}(1)} \hat{\mathbf{P}}^{\underline{u}}_{2T-1,+,\lambda}\left(\max_{t\in[0,2T-1]} x_1^\lambda(t) < 2\right).
	\end{align*}
	By \cite[Section 7.2, (\textbf{I.2}), (\textbf{I.3})]{IVW}, there is a constant $\rho>0$ so that the probability in the second to last line is bounded below by $\rho^2$ uniformly in $\underline{w}\in\mathbb{A}_{n,\lambda}^+$ with $w_1\leq\eta$ and $\lambda$ small. The probability in the last line can be bounded below by the probability that $n$ walks starting at $u_1,\dots,u_n$ stay within horizontal tubes of width $\epsilon/4$ centered at $u_1,\dots,u_n$ on the whole interval. Splitting the walks into time blocks of length $c_3\epsilon^2$, the probability that each walk stays within the corresponding tube on each time block is bounded below for small $\lambda$ by some constant $p>0$, and this gives a total lower bound of $p^{c_4 T/\epsilon^2} \geq e^{-c_5(\epsilon)T}$. Combining with \eqref{Zbd1}, we get
	\[
	Z^{\underline{w}}_{T,+,\lambda} \geq e^{-\mathfrak{a}\mathfrak{b}^{n+1}(q_0(2\eta) + (2T-1)q_0(2)) }\cdot\rho^2 e^{-c_5(\epsilon)T} \geq c_1(n) e^{-c_2(n,\eta)T}
	\]
	as desired.
	
\end{proof}

Now we finish the proof of \eqref{pregood}. Recall that $3M = a\wedge b$, so without loss of generality we may assume $a = 3M$ and $b\geq a$. In fact it suffices to consider $a,b\leq 3M$ with high probability. Indeed, as in \cite[(7.25)]{IVW}, define the random variables $B_{\pm}\geq 0$ by
\[
-2M - B_- = \max\{t\leq -2M : y_1^\lambda(t)\leq \eta\}, \quad 2M+B_+ = \min\{t\geq 2M : y_1^\lambda(t)\leq\eta\}.
\]
Then the same argument using the Gibbs property as in \cite[(7.28)]{IVW} shows that for any $b_{\pm}\geq 0$,
\begin{align*}
	\mathbb{P}^{\underline{r}}_{b,+,\lambda}(B_{\pm} = b_{\pm}) \leq c_6(\epsilon)e^{c_7(\epsilon)M - (b_+ + b_-) q_0(\eta)}.
\end{align*}
Therefore fixing $\eta$ large enough so that $q_0(\eta) > 2c_7(\epsilon)$, we can ignore the case where $b_{\pm} > M$. 

Now by the Gibbs property (conditioning on $A$, $B_{\pm}$, and the values of $\underline{x}^\lambda$ and $\underline{y}^\lambda$ at these times), it suffices to prove \eqref{pregood} in the case $a = -3M$ and $b_1,b_2\in[2M,3M]$ in place of $b$. For this we observe that if $\nu^{(5)} < 1/5$ for example, then Lemma \ref{partition} implies
\begin{align*}
	\mathbb{P}^{\underline{r}}_{a,+,\lambda}\otimes\mathbb{P}^{\underline{s}}_{b_1,b_2,+,\lambda} (\mathcal{M}_0^{(5)} \leq \nu^{(5)}M) &\leq e^{-q_0(\eta)M/10}\cdot\frac{1}{{Z}^{\underline{r}}_{6M,+,\lambda} {Z}^{\underline{u}}_{b_1+b_2,+,\lambda}}\\
	&\leq e^{-q_0(\eta)M/10} \cdot c_2(n,\eta)^{-2}e^{24 c_1(n)M}.
\end{align*}
Enlarging $\eta$ if necessary so that $q_0(\eta)/10 > 25c_1(n)$ and then taking $M$ large enough depending on $c_1(n)$ and $c_2(n,\eta)$, we obtain a constant $\kappa^{(5)}$ for which \eqref{pregood} holds. This completes the proof.
	
\bibliographystyle{plain}
\bibliography{bib}

\begin{thebibliography}{10}

\bibitem{SOS0}
J.~Bricmont, A.~El~Mellouki, and J.~Fr\"olich.
\newblock Random surfaces in statistical mechanics: {R}oughening, rounding,
  wetting...
\newblock {\em J. Stat. Phys.}, 42:743--798, 1986.

\bibitem{CG}
P.~Caputo and S.~Ganguly.
\newblock Uniqueness, mixing, and optimal tails for {B}rownian line ensembles
  with geometric area tilt.
\newblock 2023.
\newblock Preprint: arXiv:2305.18280.

\bibitem{CIW19}
P.~Caputo, D.~Ioffe, and V.~Wachtel.
\newblock Confinement of {B}rownian polymers under geometric area tilts.
\newblock {\em Electron. J. Probab.}, 24:1--21, 2019.

\bibitem{CIW2}
P.~Caputo, D.~Ioffe, and V.~Wachtel.
\newblock Tightness and line ensembles for {B}rownian polymers under geometric
  area tilts.
\newblock {\em Springer Proc. Math. Stat.}, 293:241--266, 2019.

\bibitem{SOS1}
P.~Caputo, E.~Lubetzky, F.~Martinelli, A.~Sly, and F.L. Toninelli.
\newblock The shape of the $(2+1)${D} {SOS} surface above a wall.
\newblock {\em C.R. Math.}, 350(13-14):703--706, 2012.

\bibitem{SOS2}
P.~Caputo, E.~Lubetzky, F.~Martinelli, A.~Sly, and F.L. Toninelli.
\newblock Scaling limit and cube-root fluctuations in {SOS} interfaces above a
  wall.
\newblock {\em J. Eur. Math. Soc.}, 18(5):931--995, 2016.

\bibitem{SOS3}
P.~Caputo, F.~Martinelli, and F.L. Toninelli.
\newblock On the probability of staying above a wall for the
  $(2+1)$-dimensional {SOS} model at low temperature.
\newblock {\em Probab. Theory Relat. Fields}, 163:803--831, 2015.

\bibitem{CH14}
I.~Corwin and A.~Hammond.
\newblock Brownian {G}ibbs property for {A}iry line ensembles.
\newblock {\em Invent. Math.}, 195:441--508, 2014.

\bibitem{CH16}
I.~Corwin and A.~Hammond.
\newblock {KPZ} line ensemble.
\newblock {\em Probab. Theory Relat. Fields}, 166:67--185, 2016.

\bibitem{DLZ}
A.~Dembo, E.~Lubetzky, and O.~Zeitouni.
\newblock On the limiting law of line ensembles of {B}rownian polymers with
  geometric area tilts.
\newblock 2022.
\newblock To appear in \textit{Ann. Inst. H. Poincar\'e Probab. Statist.},
  available at arXiv:2201.01635.

\bibitem{Dob}
R.L. Dobrushin.
\newblock The {G}ibbs state that describes the coexistence of phases for a
  three-dimensional {I}sing model.
\newblock {\em Teor. Verojatnost. i Primenen.}, 17:619--639, 1972.

\bibitem{FS}
P.L. Ferrari and S.~Shlosman.
\newblock The $\mathrm{Airy_2}$ process and the 3{D} {I}sing model.
\newblock {\em J. Phys. A: Math. Theor.}, 56(1):1--15, 2023.

\bibitem{GL}
R.~Gheissari and E.~Lubetzky.
\newblock Entropic repulsion of 3{D} {I}sing interfaces.
\newblock 2021.
\newblock Preprint, available at arXiv:2112.05133.

\bibitem{ISV}
D.~Ioffe, S.~Shlosman, and Y.~Velenik.
\newblock An invariance principle to {F}errari-{S}pohn diffusions.
\newblock {\em Commun. Math. Phys.}, 336:905--932, 2015.

\bibitem{IV}
D.~Ioffe and Y.~Velenik.
\newblock Low-temperature interfaces: {P}rewetting, layering, faceting, and
  {F}errari-{S}pohn diffusions.
\newblock {\em Markov Process. Relat. Fields}, 24(3):487--537, 2018.

\bibitem{IVW}
D.~Ioffe, Y.~Velenik, and V.~Wachtel.
\newblock Dyson {F}errari-{S}pohn diffusions and ordered walks under area
  tilts.
\newblock {\em Probab. Theory Relat. Fields}, 170:11--47, 2018.

\bibitem{Lig}
T.~Liggett.
\newblock An invariance principle for conditioned sums of independent random
  variables.
\newblock {\em J. Math. Mech.}, 18(6):559--570, 1968.

\bibitem{DS}
C.~Serio and E.~Dimitrov.
\newblock Uniform convergence of {D}yson {F}errari-{S}pohn diffusions to the
  {A}iry line ensemble.
\newblock 2023.
\newblock Preprint: arXiv:2305.03723.

\end{thebibliography}
	
\end{document}